\documentclass[11pt]{amsart}

 \usepackage{geometry}
\usepackage{amsmath}
\numberwithin{equation}{section}
\usepackage{graphicx}
\usepackage{amsmath}
\usepackage{amsthm}
\usepackage{amsfonts}
\usepackage{amssymb}
\usepackage{graphics}
\usepackage{enumerate}

\theoremstyle{plain}
\newtheorem{theorem}{Theorem}
 \newtheorem{satz}{Theorem}

\newtheorem{lemma}[theorem]{Lemma}
\newtheorem{corollary}[theorem]{Corollary}
\newtheorem{example}{Example}
\newtheorem{proposition}[theorem]{Proposition}
\newtheorem{remark}[theorem]{Remark}
\theoremstyle{definition}
\newtheorem{definition}[theorem]{Definition}

\newcommand\C{\mathbb C}
\newcommand\R{\mathbb R}
\newcommand\N{\mathbb N}
\newcommand\Ha{\mathbb H} 
\renewcommand\Im{\operatorname{Im}} 
\newcommand{\eps}{\varepsilon}
\newcommand{\Lip}{\operatorname{Lip}\left( \frac{1}{2}\right)}
\newcommand{\LandauO}{\mathcal{O}} 

\newcommand{\diam}{\operatorname{diam}}
\newcommand{\dist}{\operatorname{dist}}
\newcommand{\hcap}{\operatorname{hcap}}

\begin{document}

\parindent 0pt 
\setcounter{section}{0}

\author{Sebastian Schlei\ss inger}
\title[On driving functions for quasislits]{On driving functions generating quasislits in the chordal Loewner-Kufarev equation}
\date{\today}

\begin{abstract}
     We prove that for every $C>0$ there exists a driving function $U:[0,1]\to\R$ such that the corresponding chordal Loewner-Kufarev equation generates a quasislit and $ \limsup_{h\downarrow0}\frac{|U(1)-U(1-h)|}{\sqrt{h}}=C. $ 
\end{abstract}
  	\keywords{Loewner-Kufarev equation, chordal differential equation, quasislits, slit domains, driving function}
\subjclass[2010]{30C20, 30C35} 
 \email{sebastian.schleissinger@mathematik.uni-wuerzburg.de}
\maketitle

\section{Introduction and result}

Denote by $\Ha:=\{z\in\C\;|\;\Im(z)>0\}$ the upper half-plane. A bounded subset $A\subset\Ha$ is called a \emph{(compact) hull} if $A=\Ha\cap \overline{A}$ and $\Ha\setminus A$ is simply connected. By $g_A$ we denote the unique conformal
 mapping from $\Ha\setminus A$ onto $\Ha$ with \emph{hydrodynamic
   normalization}, i.e. $$g_A(z)=z+\frac{b}{z}+\LandauO(|z|^{-2}) \quad
 \text{for} \quad |z|\to\infty$$ and for some $b\geq0$. The quantity $\hcap(A):=b$ is called
 \emph{half-plane capacity} of $A.$\\

 The chordal (one-slit) Loewner-Kufarev equation for $\Ha$ is given by

\begin{equation}\label{ivp}
   \dot{g}_t(z)=\frac{2}{g_t(z)-U(t)}, \quad g_0(z)=z\in\Ha,
\end{equation}

where $U:[0,T]\to\R$ is a continuous function, the so called \emph{driving function}. For $z\in\Ha,$ let $T_z$ be the supremum of all $t$ such that the solution exists up to time $t$ and $g_t(z)\in \Ha.$ Let $K_t:=\{z\in \Ha \; | \;  T_z\leq t\},$ then  $\{K_t\}_{t\in[0,T]}$ is a family of increasing hulls and $g_t$ is the unique conformal mapping of $\Ha\setminus K_t$ onto $\Ha$ with hydrodynamic normalization and $g_t(z)=z+\frac{2t}{z}+\LandauO(|z|^{-2})$ for $z\to\infty.$\\

If $\gamma:[0,T]\to \overline{\Ha}$ is a simple curve, i.e. a continuous, one-to-one function with $\gamma(0)\in\R$ and $\gamma((0,1])\subset \Ha,$ then we call the hull $\Gamma:=\gamma((0,1])$ a \textit{slit}. The important connection between slits and equation (\ref{ivp}) is given by the following Theorem.

\begin{satz}[Kufarev, Sobolev, Spory{\v{s}}eva]\label{slitex}
 For any slit $\Gamma$ with $\hcap(\Gamma)=2T$ there exists a unique continuous driving function $U:[0,T]\to\R$ such that the solution $g_t$ of (\ref{ivp}) satisfies $g_T=g_\Gamma.$
\end{satz}
\begin{proof}
 The first proof was given by Kufarev et al. in \cite{MR0257336}. For a English reference, see \cite{Lawler:2005}, p.\,92f.
\end{proof}

Conversely, equation (\ref{ivp}) does not necessarily generate a slit for a given driving function, see Example \ref{Felix}.
A sufficient condition for getting slits was found by J. Lind, D. Marshall and S. Rohde:\\

According to \cite{MarshallRohde:2005} and \cite{Lind:2005} we define a \textit{quasislit} to be the image of $[0,i]$ under a quasiconformal mapping $Q:\Ha\to\Ha$ with $Q(\Ha)=\Ha$ and $Q(\infty)=\infty.$ In other words, a quasislit is a slit that is a quasiarc approaching $\R$ nontangentially (see Lemma 2.3 in \cite{MarshallRohde:2005}). \\

Let $\Lip$ be the set of all continuous functions $U:[0,T]\to\R$ with $$|U(t)-U(s)|\leq c\sqrt{|s-t|} \quad \text{for some}\quad c>0 \quad \text{and all} \quad s,t\in[0,T].$$
Now the following connection between $\Lip$ and quasislits holds.

\begin{satz}\label{lmr}(Theorem 1.1 in \cite{MarshallRohde:2005} and Theorem 2 in \cite{Lind:2005})
 If $\Gamma$ is a quasislit with driving function $U$, then $U \in \Lip$. Conversely, if $U \in \Lip$ and for every $t\in[0,T]$ there exists an $\eps>0$  such that
$$ \sup_{\substack{r,s\in[0,T]\\ |r-t|,|s-t|<\eps}}\frac{|U(r)-U(s)|}{\sqrt{|r-s|}}<4, $$
then $\Gamma$ is a quasislit.
\end{satz}

The H\"older constant $4$ in Theorem \ref{lmr} is not necessary for generating quasislits: For any $s\in[0,T)$, the ``right pointwise H\"older norm'', i.e. the value $$\limsup_{h\downarrow0}\frac{|U(s+h)-U(s)|}{\sqrt{h}}$$ can get arbitrarily large, as the driving function $U(t)=c\sqrt{t}$ shows:

\begin{example}\label{Maxime}
Let $U(t)=c\sqrt{t}$ for an arbitrary $c\in\R.$ In this case, the one-slit equation (\ref{ivp}) can be solved explicitly and one obtains for the generated hull $K_t$ at time $t:$ $K_t=\gamma([0,t])$ with $\gamma(t)=2\sqrt{t}\left(\frac{\pi}{\phi}-1\right)^{\frac1{2}-\frac{\phi}{\pi}}e^{i\phi},$ i.e. $K_t$ is a line segment with angle $\phi,$ see Example 4.12 in \cite{Lawler:2005}. The connection between $c$ and $\phi$ is given by
$$c(\phi)=\frac{2(\pi-2\phi)}{\sqrt{\phi(\pi-\phi)}},\qquad \phi(c)=\frac{\pi}{2}\left(1-\frac{c}{\sqrt{c^2+16}}\right).$$
\end{example}

The aim of this paper is to show that the ``left pointwise H\"older-norm'' can also get arbitrarily large within the space of all driving functions that generate quasislits.

\begin{theorem}\label{proop}
 For every $C>0,$ there exists a driving function $U:[0,1]\to\R$ that generates a quasislit and satisfies
$$ \limsup_{h\downarrow0}\frac{|U(1)-U(1-h)|}{\sqrt{h}}=C. $$
\end{theorem}

We will prove Theorem \ref{proop} in Section 3. In Section 2 we discuss
some necessary and sufficient conditions on driving functions for generating  slits.

\section{Conformal Welding in the Loewner equation}

In the following we denote by $B(z,r),$ $z\in\C, r>0$, the Euclidean ball with centre $z$ and radius $r$, i.e. $B(z,r)=\{w\in\C\;|\; |z-w|<r\}.$\\

As already mentioned, there are continuous driving functions that generate hulls which are not slits.

\begin{example}\label{Felix}
 Consider the driving function $U(t) = c\sqrt{1-t}$ with $c\geq 4$ and $t\in [0,1]$: The generated hull $K_t$ is a slit for every $t\in(0,1),$ but at $t=1$ this slit hits the real axis at an angle $\varphi$ which can be calculated directly (see \cite{LindMR:2010}, chapter 3): $$ \varphi = \pi-\frac{2\pi\sqrt{c^2-16}}{\sqrt{c^2-16}+c}. $$ Its complement with respect to $\Ha$ has two connected components and $K_1$ is the closure of the bounded component  and consequently not a slit, see Figure 1.
\end{example}

\begin{figure}[h] 
    \centering
   \includegraphics[width=125mm]{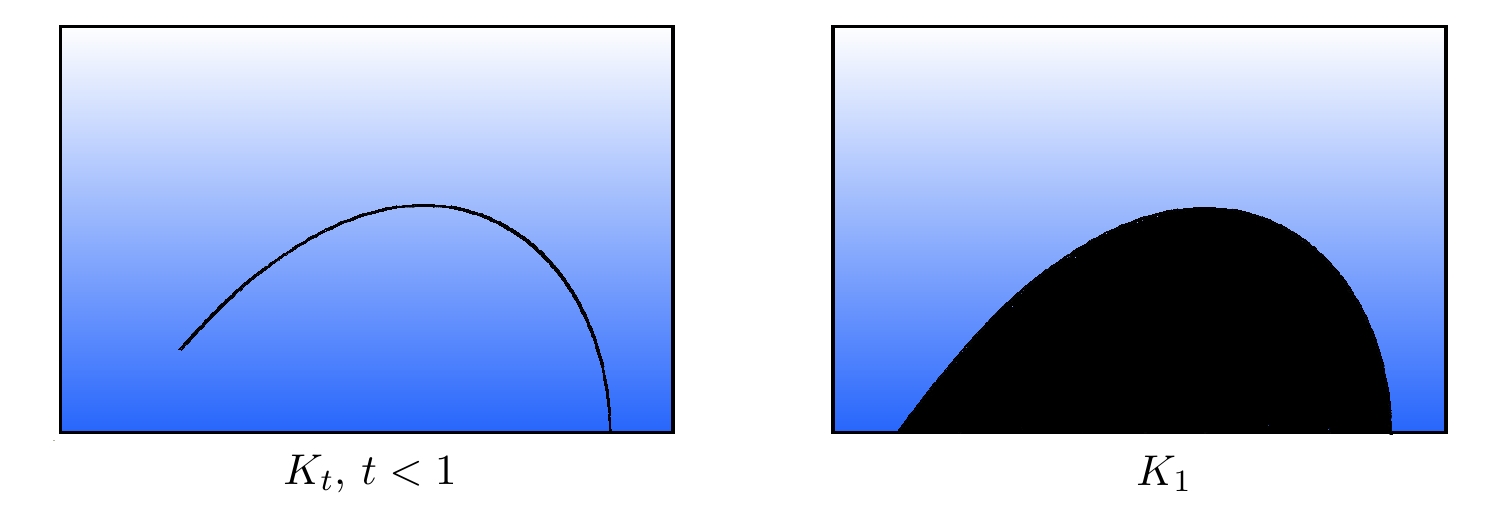}
\caption{Example \ref{Felix} with $c=5.$}
 \end{figure}
There are further, more subtle obstacles preventing the one-slit equation from producing slits as the following example shows.

\begin{example}\label{spiral}
There exists a driving function $U\in \Lip$ such that $K_t$ is a simple curve $\gamma$ for all $t\in [0,T)$ and for $t\to T$ this curve wraps infinitely often around $B(2i,1).$ Hence $K_T=\gamma[0,1)\cup \overline{B(2i,1)}$ is not locally connected, see Example 4.28 in \cite{Lawler:2005}.
\end{example}

In order to distinguish between these two kinds of obstacles, one has introduced two further notions, which are more general than ``the hull is a slit''.

For this we have to take a look at the so called backward equation. Let $U:[0,T]\to\R$ be continuous. Furthermore, let $g_t$ be the solution to (\ref{ivp}). The backward equation is given by
\begin{equation}\label{back}
 \dot{x}(t) = \frac{-2}{x(t)-U(T-t)}, \quad x(0)=x_0\in\overline{\Ha}\setminus\{U(T)\}.
\end{equation}
For $x_0\in \Ha,$ the solution $x(t)$ exists for all $t\in[0,T]$ and the function $f_T:x_0\mapsto x(T)$ satisfies $f_T=g_T^{-1}.$

For $x_0\in \R\setminus\{U(T)\},$ the solution may not exist for all $t\in[0,T].$ If a solution ceases to exist, say at $t=s$, it will hit the singularity, i.e. $\lim_{t\to s}x(t)=U(T-s).$

Now suppose that two different solutions $x(t),y(t)$ with $x(0)=x_0<y_0=y(0)$ meet the singularity at $t=s$, i.e. $\lim_{t\to s}x(s)=\lim_{t\to s}y(s)=U(T-s)$. Then $x_0$ and $y_0$ lie on different sides  with respect to $U(T)$, that is $x_0<U(T)<y_0.$ Otherwise the difference $y(t)-x(t)$ would satisfy 
$$ \dot{y}(t)-\dot{x}(t)= \frac{2(T-t)(y(t)-x(t))}{(y(t)-U(T-t))(x(t)-U(T-t))}>0$$
for all $0\leq t<s$ and thus, $\lim_{t\to s}(x(t)-y(t))=0$ would be impossible.

Consequently, for any $s\in(0,T]$, there are at most two initial values so that the corresponding solutions will meet in $U(T-s)$.

\begin{definition} 
Let $\{K_t\}_{t\in[0,T]}$ be a family of hulls generated by the one-slit equation with driving function $U$.
\begin{enumerate}
 \item[(1)] $\{K_t\}_{t\in[0,T]}$ \emph{is welded} if for every $s\in(0,T]$ there exist exactly two real values $x_0,y_0$ with $x_0<U(T)<y_0$ such that the corresponding solutions $x(t)$ and $y(t)$ of (\ref{back}) with $x(0)=x_0,$ $y(0)=y_0$ satisfy $x(s)=y(s)=U(T-s).$\\[-0.2cm]
\item[(2)] $\{K_t\}_{t\in[0,T]}$ is \emph{generated by a curve} if there exists a simple curve $\gamma:[0,T]\to\overline{\Ha},$ such that $\Ha\setminus K_t$ is the unbounded component of $\Ha\setminus \gamma[0,t]$ for every $t\in[0,T].$

\end{enumerate}
\end{definition}

\begin{remark}
 Several properties of hulls that are generated by curves are described in \cite{Lawler:2005}, Section 4.4. The notion of welded hulls was introduced in \cite{MarshallRohde:2005} for the radial Loewner equation. The chordal case is considered in \cite{Lind:2005}. Informally speaking, welded hulls have a left and a right side.

If $\{K_t\}_{t\in[0,T]}$ is welded and the interval $I=[a,b]$ is the cluster set of $K_T$ with respect to $g_T$, then $a<U(T)<b$ and for every $a\leq x_0<U(T)$ there exists $U(T)<y_0\leq b$ such that the solutions to (\ref{back}) with initial values $x_0$ and $y_0$ hit the singularity at the same time. This gives a welding homeomorphism $h:[a,b]\to[a,b]$ by defining $h(x_0):=y_0,$ $h(y_0):=x_0,$ $h(U(T)):=U(T).$

Furthermore, the hulls $\{K_t\}_{t\in[0,T]}$ describe a quasislit if and only if they describe a slit and the homeomorphism $h$ is a quasisymmetric function, see Lemma 6 in \cite{Lind:2005}.
\end{remark}

The hull of Example \ref{spiral}, which is sketched in picture a) of Figure 2, is not generated by a curve. Picture c) shows an  example of a hull that is generated by a curve. Here, the curve hits itself and the real axis and consequently, this hull is not welded. The hulls in picture b), which form a ``topologist's sine curve'' approaching a compact interval on $\R$, are simple curves before the ``sine curve'' touches $\R$, but then, the corresponding hull is neither welded nor generated by a curve. 
\begin{figure}[h]
 \centering \includegraphics[width=125mm]{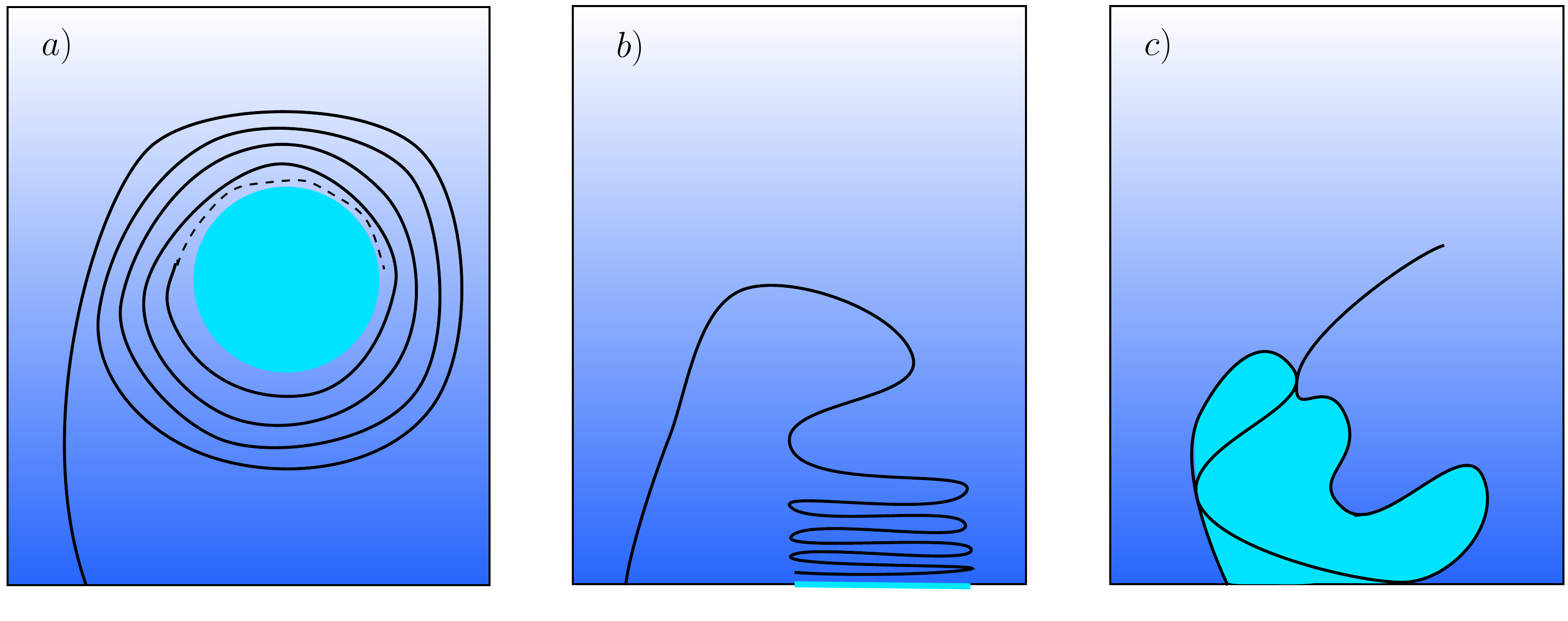}
\caption{Three cases of hulls that are not slits.}
\end{figure}

\begin{proposition} 
The hulls $\{K_t\}_{t\in[0,T]}$ describe a slit if and only if $\{K_t\}_{t\in[0,T]}$ is generated by a curve and it is welded.
\end{proposition}
\begin{proof}
For the non-trivial direction of the statement, see Lemma 4.34 in \cite{Lawler:2005}.
\end{proof}

The following statement follows directly from the proof of Lemma 3 in \cite{Lind:2005}. For the sake of completeness we include the proof.

\begin{proposition}\label{Fritz} Let $\{K_t\}_{t\in[0,T]}$ be family of hulls generated by the one-slit equation. The following statements are equivalent:
\begin{itemize}
 \item[a)] $\{K_t\}_{t\in[0,T]}$ is welded.
\item[b)] For every $\tau\in[0,T)$ there exists $\eps>0$ such that for all $x_0\in\R\setminus\{U(\tau)\}$ the solution $x(t)$ of 
\begin{equation}\label{resp} \dot{x}(t)=\frac{2}{x(t)-U(t)}, \qquad x(\tau)=x_0, \end{equation} does not hit $U(t)$ for $t<T$ and satisfies $|x(T)-U(T)|>\eps$. 
\end{itemize}
\end{proposition}
\begin{proof}
$a)\Longrightarrow b):$ Firstly, the solution $x(t)$ to (\ref{resp}) exists locally, say in the interval $[\tau, T^*),$ $T^*\leq T.$ 
Now we know that there are $x_1^0,x_2^0$ with $x_1^0<x_2^0,$ such that the solutions $x_1(t)$ and $x_2(t)$ to equation (\ref{back}) with initial values $x_1^0$ and $x_2^0$ respectively hit the singularity $U(\tau)$ at $s=T-\tau$. But this implies that $x(t)$ can be extended to the interval $[0,T^*]$ with $|U(T^*)-x(T^*)|<\eps,$ where $\eps:=\min\{U(T^*)-x_1(T-T^*),x_2(T-T^*)-U(T^*)\}.$

$b)\Longrightarrow a):$ Let $\tau=T-s.$ We set $a_n:=U(\tau)-\frac{1}{n}$ for all $n\in\N.$ The solution $x_n(t)$ of (\ref{resp}) with initial value $a_n$ exists up to time $T$. Hence we can define $\xi_n:=x_n(T)$ for all $n\geq N$ and we have $U(T)-\xi_n>\eps.$ The sequence $\xi_n$ is increasing and bounded above, and so it has a limit $x_0<U(T).$ Then the solution $x(t)$ of (\ref{back}) with $x_0$ as initial value satisfies $\displaystyle \lim_{t\to s}x(t)=\lim_{n\to\infty}a_n=U(\tau)=U(T-s).$

The second value $y_0$ can be obtained in the same way by considering the sequence $U(\tau)+\frac{1}{n}$ instead of $a_n.$
\end{proof}


\begin{remark} The proof of Proposition 3.1 in \cite{MR2920433} implies the following necessary condition for getting welded hulls:
 If $K_T$ is welded, then, for every $s\in(0,T],$ we have

{ \small \begin{eqnarray*}
 \limsup_{h\downarrow0}\frac{|U(s)-U(s-h)|}{\sqrt{h}}<& 4& \quad \;\;(\text{\bf ``regular case''}), \;\; \text{or}    \\     
 \liminf_{h\downarrow0}\frac{|U(s)-U(s-h)|}{\sqrt{h}} <& 4& \leq \limsup_{h\downarrow0}\frac{|U(s)-U(s-h)|}{\sqrt{h}} \;\;({\text{\bf ``irregular  case''}}). \end{eqnarray*} }

Conversely, if the regular condition holds for every $s\in(0,T],$ then $K_T$ is welded, see \cite{Lind:2005}, Section 5.
\end{remark}

For driving functions which are ``irregular'' in at least one point, it is somehow harder to find out whether the generated hulls are welded or not. Here we derive a sufficient condition for a very special case. This case will be needed in the proof of Theorem \ref{proop}.

\begin{lemma}\label{stran}
Let $U:[0,1]\to\R$ be continuous with $U(1)=0$ and let $\{K_t\}_{t\in[0,1]}$ be the hulls generated by equation (\ref{ivp}). Suppose that there are two increasing sequences $s_n,t_n$ of positive numbers with $s_n,t_n\to1, 
$ such that for $$\overline{M}_n:=\max_{s_n\leq t\leq1}\{U(t)\}\quad \text{and} \quad \underline{M}_n:=\min_{t_n\leq t\leq1}\{U(t)\}$$ the two inequalities 
$$4(1-s_n)+U(s_n)^2-2U(s_n)\overline{M}_n>0 \quad \text{and} \quad 4(1-t_n)+U(t_n)^2-2U(t_n)\underline{M}_n>0$$
 hold for all $n\in\N.$ If $K_t$ is welded for all $t\in(0,1)$, then so is $K_1$.
\end{lemma}
\begin{proof}
Let $\tau\in[0,1)$ and $x_0\in\R\setminus\{U(\tau)\}.$
By Proposition \ref{Fritz} we know that the solution $x(t)$ of the initial value problem
 $$ x(\tau)=x_0, \qquad \dot{x}(t)=\frac{2}{x(t)-U(t)}$$
 exists until $t=1$ and we have to show that there is a lower bound for $|x(1)-U(1)|$ which is independent of $x_0$.

 Assume that $x_0<U(\tau)$. Then $x(t)$ is decreasing and we have $x(s_m)<U(s_m)$ with $m:=\min\{n\in\N \; |\; s_n\geq \tau\}.$ The initial value problem 
$$ \dot{y}(t)=\frac{2}{y(t)-\overline{M}_m}, \qquad  y(s_m)=x(s_m),$$
has the solution $y(t)=\overline{M}_m-\sqrt{(M_m-x(s_m))^2+4(t-s_m)}$. Now we have
 $$\dot{x}(t)\leq\frac{2}{x(t)-\overline{M}_m} \quad \text {for all} \; t\in [s_m,1)$$ and $x(s_m)=y(s_m)$. Consequently,
 \begin{eqnarray*}
  &&x(1)\leq y(1)=\overline{M}_m-\sqrt{(\overline{M}_m-x(s_m))^2+4(1-s_m)}\\
&&<\underbrace{\overline{M}_m-\sqrt{(\overline{M}_m-U(s_m))^2+4(1-s_m)}}_{=:L_1}<\overline{M}_m-\sqrt{\overline{M}_m^2}=0.
 \end{eqnarray*}
The case $x_0>U(\tau)$ can be treated in the same way and gives a bound $L_2>0$ for $x(1)=x(1)-U(1).$
Thus, the condition in Proposition \ref{Fritz} b) is satisfied for $\eps=\min\{L_1,L_2\}$ and it follows that $K_1$ is welded.
\end{proof}
\begin{corollary}\label{cori}
 If $K_t$ is welded for all $t\in(0,1)$ and there are two increasing sequences $s_n,t_n$ of positive numbers with $s_n, t_n\to1,$ and $U(s_n)\leq U(1) \leq U(t_n)$ for all $n,$ then $K_1$ is welded, too.
\end{corollary}
\begin{proof}
Without loss of generality we can assume $U(1)=0.$ We can apply Lemma \ref{stran} as  $$4(1-s_n)+U(s_n)^2-2U(s_n)\overline{M}_n>-2U(s_n)\overline{M}_n\geq0 \quad \text{and}$$ 
$$ 4(1-t_n)+U(t_n)^2-2U(t_n)\underline{M}_n>-2U(t_n)\underline{M}_n\geq0.$$
\end{proof}

\section{Proof of Theorem \ref{proop}}

Now we are ready to prove Theorem \ref{proop}.

\begin{proof}[Proof of Theorem \ref{proop}]
Let $C>0.$ First we construct the driving function $U$, which is shown in Figure 3 for the case $C=5$. 

We set $U(r_n):=0$ with $r_n:=1-\frac1{2^n}$ for all $n\geq 0.$ The mean value of $r_n$ and $r_{n+1}$ is equal to $w_n:=1-\frac3{2^{n+2}}$ and here we define $$U(w_n) := C\sqrt{\frac3{2^{n+2}}}\quad \text{for} \quad n\geq0.$$ Now we define $U(t)$ for $t\in[0,1)$ by linear interpolation, so that 
\begin{eqnarray*}
 U(t)&=& C\sqrt{3\cdot2^{n+2}}\cdot (t-r_n) \quad \text{for} \quad t\in[r_n, w_n]\quad \text{and}\\
U(t)&=& C\sqrt{3\cdot2^{n+2}}\cdot (r_{n+1}-t) \quad \text{for} \quad t\in[w_n,r_{n+1}].
\end{eqnarray*}
  By defining $U(1):=0$ we now have a continuous driving function and 
$$ \limsup_{h\downarrow0}\frac{|U(1)-U(1-h)|}{\sqrt{h}}=
\lim_{n\to\infty}\frac{|U(1)-U(w_n)|}{\sqrt{1-w_n}} = 
\lim_{n\to\infty}\frac{C\sqrt{3/2^{n+2}}}{\sqrt{3/2^{n+2}}} = C.$$
At each $0\leq t <1,$ the hull $K_t$ produced by this function will be a quasislit according to Theorem \ref{lmr}. Thus, we have to show that also $K_1$ is a slit and that this slit is a quasiarc.\\

First, we know that $\{K_t\}_{t\in[0,1]}$ is welded: This follows directly from  Corollary \ref{cori} by setting $s_n:=t_n:=r_n$.

If we scale our hull by $\frac1{\sqrt{2}}$, we end up with the new driving function $\tilde{U}:[0,1/2]\to\R,$ $\tilde{U}(t)=\frac1{\sqrt{2}}U(2t)$. However, this is again $U(t)$, confined to the interval $[1/2,1],$ i.e. $\tilde{U}(t)=U(t+1/2).$ Geometrically, this means that $g_{1/2}(\overline{K_1\setminus K_{1/2}})$ is just the same as $\frac1{\sqrt{2}}K_1,$ the original hull scaled by $\frac1{\sqrt{2}}.$\\

If $f:=g_{1/2}^{-1}$, and $S_n:=\overline{K_{1-1/2^n}\setminus K_{1-1/2^{n-1}}}, n\geq1,$ then  we have  $$S_{n+1}=f\left(\frac1{\sqrt{2}}S_n\right).$$
As the function $z\mapsto f(\frac{1}{\sqrt{2}}z)=:I(z)$ is not an automorphism of $\Ha$, the Denjoy--Wolff Theorem implies that the iterates $I^n=(I\circ\ldots\circ I)$ converge uniformly on $S_1$ to a point $S_\infty \in \overline{\Ha}\cup\{\infty\}.$ $S_\infty=\infty$ is not possible as the hull $K_1$ is a compact set and the case $S_\infty\in \R$ would imply that $K_1$ is not welded. Consequently $S_\infty\in \Ha$ and $K_1=\displaystyle \bigcup_{n\geq 1} S_n \cup \{S_\infty\}$ is a simple curve whose tip is $S_\infty$.\\

Now we show that this curve is a quasiarc.

For this, we will use the metric characterization of quasiarcs by Ahlfors' three point property (also called \emph{bounded turning property}, see \cite{MR0344463}, Section 8.9, or \cite{MR0210889}, Theorem 1), which says that $K_1$ is a quasiarc if and only if 
$$ \sup_{
\substack{x,y \in K_1 \\ x\not=y}} \frac{\diam(x,y)}{|x-y|} < \infty,$$ 
where we denote by $\diam(x,y)$ the diameter of the subcurve of $K_1$ joining $x$ and $y.$

For $m\in\N\cup\{0\}$ we define $F_m:=\bigcup_{k\geq m}S_{k}\cup\{S_\infty\}$. As $K_t$ is a quasislit for every $t\in(0,1),$ it suffices to show that 
\begin{equation}\label{show}
 \sup_{
\substack{x,y \in F_m \\ x\not=y}} \frac{\diam(x,y)}{|x-y|} < \infty \quad \text{for one} \; m \in\N.
\end{equation}

The set $S_n$ contracts to $S_\infty$ when $n\to\infty,$ in particular $\diam(S_n)\to0.$ 

As $I$ is conformal in $B(S_\infty, \eps)$ for $\eps>0$ small enough, there is an $N=N(\eps)\in\N$, such that 
$S_n\subset B(S_\infty, \eps)$ for all $n\geq N$. 

$S_\infty$ is a fixpoint of $I(z)$ and so $|I'(S_\infty)|<1.$ Otherwise, $I$ would be an automorphism of $\Ha.$\\

Now, for $x\in S_{N+n}$, $n\geq 0,$ we have $I(x)\in S_{N+n+1}$ and
 \begin{eqnarray*}
 &&|I(x)-S_\infty|=|I(x)-I(S_\infty)|=|I'(S_\infty)(x-S_\infty)+\LandauO(|x-S_\infty|^2)|\\&=&
|I'(S_\infty)+\LandauO(|x-S_\infty|)|\cdot|x-S_\infty|=|I'(S_\infty)||1+\LandauO(|\eps|)|\cdot|x-S_\infty|.
\end{eqnarray*} Consequently we can pass on to a smaller $\eps$ (and larger $N$) such that $\dist(S_\infty, S_{N+n+1})\leq c\dist(S_\infty, S_{N+n})$ with $c<1$ and for all $n\geq0$. Hence 
$S_{N+n}\subset B(S_\infty,c^n\eps).$\\

Furthermore, for $x,y\in F_{N+n}$ we have
\begin{eqnarray*}
 |I(x)-I(y)|= |I'(x)+\LandauO(c^n\eps)|\cdot|x-y|=|I'(S_\infty)+\LandauO(c^n\eps)|\cdot|x-y|.
\end{eqnarray*}

Hence there exist positive constants $a_1,a_2$ with $1-a_2c^n>0$ such that
\begin{equation}\label{diam}(1-a_2c^n)|I'(S_\infty)||x-y|\leq|I(x)-I(y)|\leq(1+a_1c^n)|I'(S_\infty)||x-y|.\end{equation} Thus \begin{equation}\label{diam2}\diam(I(x),I(y))\leq (1+a_1c^n)|I'(S_\infty)|\diam(x,y).\end{equation}
 
Now we will show (\ref{show}) for $m=N.$ Let $x, y \in F_N$ with $x\not=y.$ We assume that $\diam(x,S_\infty)\geq \diam(y,S_\infty).$ In particular, $x\not=S_\infty$ and thus there is a $k\geq 0$ and an $\hat{x}\in S_N$ such that $x=I^k(\hat{x}).$ Let $\hat{y}\in F_N$ be defined by $y=I^k(\hat{y}).$ 
First note that
$$ \sup_{
\substack{a\in S_N, b\in F_N \\ a\not=b}} \frac{\diam(a,b)}{|a-b|}=:C<\infty,
$$
for $K_t$ is a quasislit for every $t\in(0,1).$ Now we get with (\ref{diam}) and (\ref{diam2}):

\begin{eqnarray*}
 \frac{\diam(x,y)}{|x-y|} &=&  \frac{\diam(I^k(\hat{x}),I^k(\hat{y}))}{|I^k(\hat{x})-I^k(\hat{y})|}\leq \frac{(1+a_1c^{k-1})|I'(S_\infty)|\diam(I^{k-1}(\hat{x}),I^{k-1}(\hat{y}))}{(1-a_2c^{k-1})|I'(S_\infty)||I^{k-1}(\hat{x})-I^{k-1}(\hat{y})|} =\\
&=& \frac{(1+a_1c^{k-1})}{(1-a_2c^{k-1})}\cdot \frac{\diam(I^{k-1}(\hat{x}),I^{k-1}(\hat{y}))}{|I^{k-1}(\hat{x})-I^{k-1}(\hat{y})|} \leq ... \leq \\
&\leq& \prod \limits_{j=0}^{k-1} \frac{(1+a_1c^j)}{(1-a_2c^j)}\cdot \frac{\diam(\hat{x},\hat{y})}{|\hat{x}-\hat{y}|}  \leq
 \prod \limits_{j=0}^{k-1} \frac{(1+a_1c^j)}{(1-a_2c^j)} \cdot C \leq   C \prod  \limits_{j=0}^{\infty}\frac{1+a_1c^j}{1-a_2c^j} =\\ &=&
 C  \prod \limits_{j=0}^{\infty}(1+a_1c^j) / \prod \limits_{j=0}^{\infty}(1-a_2c^j)   <\infty.
\end{eqnarray*}

The two Pochhammer products converge because $|c|<1.$ Consequently, $K_1$ is a quasislit.
\end{proof}

  \begin{figure}[h]
    \centering
\includegraphics[width=130mm]{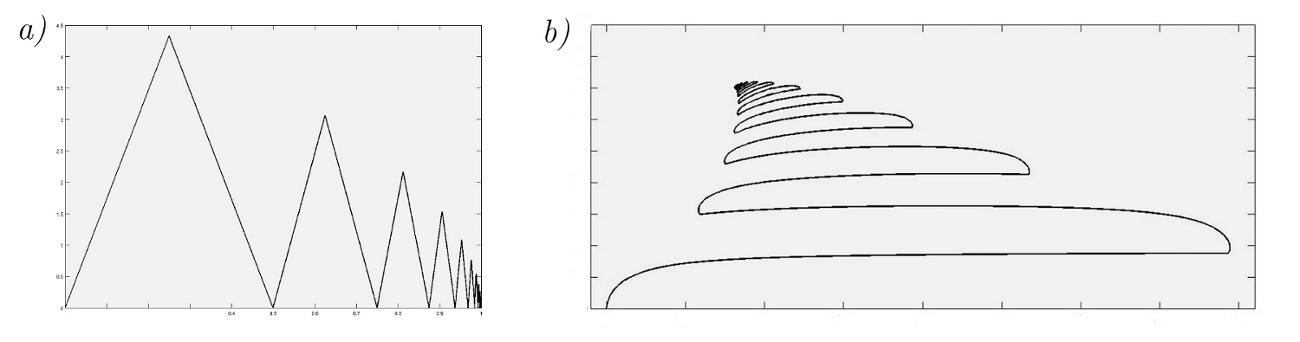}
 \caption{The driving function $U$ from Proposition \ref{proop} with $C=5$ (left) and the generated quasislit (right).}
 \end{figure}

\begin{remark}
 The argument that $U$ from the proof of Proposition \ref{proop} generates a slit holds for a more general case:

 Let $U:[0,1]\to\R$ be continuous with $U(1)=0.$ Call such a function $d$--similar with $0<d<1$ if $V(t):=U(1-t)$ satisfies $$d\cdot V(t/d^2)=V(t) \quad \text{for all}\quad 0<t\leq d^2.$$ Every $d$--similar function can  be constructed by defining $V(1)$ arbitrarily, putting $V(d^2)=d\cdot V(1)$ and then defining $V(t)$ for $d^2<t<1$ such that $V$ is continuous in $[d^2,1].$ Then, $V$ is uniquely determined for all $0\leq t \leq 1.$ Now we have: \emph{If $U$ is $d$--similar such that it produces a slit for all $0\leq t<1$ and the hull at $t=1$ is welded, then $K_1$ is a slit, too. }
\end{remark}

\begin{example}
 There exists a driving function $V:[0,1]\to\R$ that is ''irregular`` at infinitely many points and generates a quasislit:\\

 Let $U$ be the driving function from the proof of Proposition \ref{proop}. We construct  $V:[0,1]\to\R$ by sticking pieces of $U$ appropriately together. For $n\geq0$ let
$$ V(t):=  U((t-(1-1/2^n))\cdot2^{n+1}) / \sqrt{2^n}
  \quad\text{for}\quad t\in[1-1/2^n, 1-1/2^{n+1}], $$ 
and $V(1):=0.$ Then $V$ is ''irregular`` at $1-1/2^n$ for all $n\geq1$ and it produces a quasislit: The hull generated at $t=1/2$ is a quasislit by Theorem \ref{proop}. Now one can repeat the proof of Theorem \ref{proop} to show that the whole hull is a quasislit, too.
\end{example}

These examples together with Theorem \ref{lmr} suggest the following question:
 Does $U$ generate a quasislit if $U$ generates a slit and $U\in \Lip$?

The answer is no: There are $\Lip$-driving functions that generate slits with positive area. These slits cannot be quasislits as they are not uniquely determined by their welding homeomorphisms, see Corollary 1.4 in \cite{spacefill}.

\bibliographystyle{amsalpha}

\end{document}